\documentclass[12pt]{article}

\usepackage{amsfonts,amssymb,amsmath,comment,mathtools,amsthm}
\usepackage{graphicx,graphics,color}
\usepackage{caption}
\usepackage{subcaption}
\usepackage{pgf,tikz}
\usetikzlibrary{arrows}
\usepackage{enumerate}
\usepackage{cancel}
\usepackage{hyperref}  
\usepackage{mathrsfs}

\usepackage{bbm}
\usepackage{geometry}

\newtheorem{theorem}{Theorem}[section]
\newtheorem{lemma}[theorem]{Lemma}
\newtheorem{proposition}[theorem]{Proposition}

\begin{document}

\title{\large\textbf{THE RANDOM CLUSTER MODEL ON THE COMPLETE GRAPH VIA LARGE DEVIATIONS}} 
\date{} 
\author{Darion Mayes}
\maketitle

\begin{abstract}
We study the emergence of the giant component in the random cluster model on the complete graph, which was first studied by Bollobás, Grimmett, and Janson in \cite{BGJ1996}. We give an alternative analysis using a thermodynamic/large deviations approach introduced by Biskup, Chayes, and Smith in \cite{BCS2007} for the case of percolation. In particular, we compute the rate function for large deviations of the size of the largest connected component of the random graph for $q\geq 1$.
\end{abstract}

\section{Introduction}

The random cluster model was introduced by Fortuin and Kasteleyn in \cite{FK1972} as a generalisation of several existing models in statistical physics. For a finite, simple graph $G=(V,E)$, the random cluster model with \textit{edge weight} $p\in[0,1]$ and \textit{cluster weight} $q>0$ is the probability measure $\phi_{G,p,q}$ on the state space $\Omega_G=\{0,1\}^E$ defined by
\begin{equation}\label{Equation Definition of random cluster model}
\phi_{G,p,q}[\omega]:=\dfrac{\bigl\{\prod_{e \in E}p^{\omega_e}(1-p)^{1-\omega_e}\bigr\}q^{k(\omega)}}{Z_{G,p,q}},\ \omega=(\omega_e)_{e\in E} \in\Omega_G,
\end{equation}
where
\begin{equation}\label{Equation Definition of Z}
Z_{G,p,q}:=\sum_{\omega\in\Omega_G}\bigl\{\prod_{e \in E}p^{\omega_e}(1-p)^{1-\omega_e}\bigr\}q^{k(\omega)}
\end{equation}
is the normalising partition function, and $k(\omega)$ is the number of connected components in the graph $G(\omega)=(V, E(\omega))$ with $E(\omega)=\{e\in E:\ \omega_e=1\}$. For an element $\omega\in\Omega_G$, we say that the edge $e$ is \textit{open} if $e\in E(\omega)$, otherwise, it is \textit{closed}. The measure $\phi_{G,p,q}$ may be equivalently interpreted as a probability distribution on the set of random edge subgraphs $G(\omega)=(V, E(\omega))$ of $G$. For $q=1$, we recover the percolation model, for which we use the standard shorthand $\phi_{G,p}$.

We will be interested in the particular case of the complete graph $K_n=(V_n, E_n)$ when we choose $\lambda\in[0,\infty)$ and rescale the edge weight as $p=\lambda/n$. For the corresponding random cluster measure, we use the shorthand $\phi_{n,\lambda,q}:= \phi_{K_n,\lambda/n,q}$. Similarly, we denote the partition function as $Z_{n,\lambda,q}$, and use the shorthand $\phi_{n,\lambda}$ in the case of percolation. For fixed $q\geq 1$, the random cluster model  is stochastically ordered in $\lambda$, and we study how the random graph obtained under the measure $\phi_{n,\lambda,q}$ varies with respect to this parameter. 

One particular quantity of interest is the size of the largest connected component in the random graph. In their groundbreaking paper \cite{ER1960}, Erdős and Rényi studied the measure $\phi_{n,\lambda}$ and showed the existence of a critical parameter $\lambda_c=1$ marking the emergence of a \textit{giant} component containing a positive proportion of the total number of vertices asymptotically almost surely. More specifically, they showed that, with probability tending to one as $n\to\infty$ under the percolation measure $\phi_{n,\lambda}$, the largest component of the graph $K_n(\omega)=(V_n,E_n(\omega))$ is of order $\log n$ for $\lambda <1$ and of order $n$ for $\lambda >1$. This abrupt change in the size of the largest component (in the limit as $n\to\infty$) is known as an \textit{asymptotic phase transition}.

The main idea in the proof of the asymptotic phase transition of \cite{ER1960} is an exploration process whereby one chooses a vertex and sequentially inspects which vertices are connected to it in the random graph. Provided one has not yet explored a large fraction of the vertices, this exploration can be approximated by a Poisson branching process with mean $\lambda$ in the limit as $n\to\infty$. The emergence of a giant component then corresponds to the \textit{survival} of the corresponding Poisson branching process. This method is not crucially  dependent on the structure of the complete graph, and has been successfully applied to establish similar asymptotic phase transitions on a variety of families of graphs, such as the hypercube in \cite{AKS1982}, and expander graphs in \cite{ABS2004}.

For $q\neq 1$, the presence of the factor $q^{k(\omega)}$ in the random cluster measure $\phi_{n,\lambda,q}$ creates an additional complexity---the states of the edges in the random graph are no longer independent and it is no longer clear that the exploration process may be used. Nevertheless,  it has been established by Bollobás, Grimmett, and Janson in \cite{BGJ1996}  that the largest component of the complete graph  undergoes an asymptotic phase transition at a particular value $\lambda_c(q)$. The claim is proven by randomly colouring the vertices using two colours (say red and green) in a particular way so that the distribution of edges on the \textit{red vertices} is given by a percolation measure, to which we may apply the exploration process used in \cite{ER1960}. Note that the colouring argument depends crucially on the fact that for a fixed set $R$ of red vertices, the conditional distribution of edges yields a random percolation on the complete graph induced on the set $R$. This is a particular fact that does not generalise to more structured families of graphs.

The purpose of the present paper is to provide a new analysis of the random cluster model $\phi_{n,\lambda,q}$  on the complete graph, with the hope that it may be extended to wider families of graphs. Our analysis employs the methods introduced by Biskup, Chayes and Smith in \cite{BCS2007}, who analysed the largest component for percolation on the complete graph using a large deviations approach instead of branching processes. We extend this approach to the random cluster model on the complete graph for $q \geq 1$. In particular, we compute the rate function for large deviations of the size of the largest connected component, thereby recovering the asymptotic phase transition proven in \cite{BGJ1996}. In addition, we obtain a limit for the \textit{free energy} of the random cluster model. This was also computed in \cite{BGJ1996} via the colouring argument, but used to study the large deviations of the \textit{number} of connected components, rather than their size. As a byproduct of our analysis, we also obtain the exponential decay rate for the events that the random graph is connected and acyclic, respectively.

\section{Statement of Results}

In this section, we state the main results of the paper. Fix $q>0$ and $\lambda>0$. We use $\phi_{n,\lambda,q}$ to denote the random cluster probability measure on the complete graph $K_n$ with edge weight $p=\lambda/n$ and cluster weight $q$. Similarly, we denote the corresponding partition function by $Z_{n,\lambda,q}$. In some of our results, we will analyse the measure of an event $A\subset\Omega_{K_n}$ before normalisation, in which case we write $Z_{n,\lambda,q}[A]:=Z_{n,\lambda,q}\phi_{n,\lambda,q}[A]$. When $q=1$, we drop the subscript $q$ entirely and use the standard percolation notation $\phi_{n,\lambda}$.

Let $\mathcal{N}_r$ be the number of connected components in $K_n(\omega)$ of size larger than $r$, and let $\mathcal{V}_r$ be the set of vertices in $K_n(\omega)$ belonging to components of size larger than $r$. In addition, recall the \textit{entropy function}, defined for $\theta\in(0,1)$ by
\begin{equation}\label{Equation Entropy function}
S(\theta)=\theta\log\theta+(1-\theta)\log(1-\theta).
\end{equation}
Following \cite{BCS2007}, we define two functions on $[0,\infty)$ by
\begin{equation}\label{Equation Two functions}
\pi_1(x)=1-e^{-x},\ \Psi(x)=\bigg(\log x-\frac{1}{2}\bigg[x-\frac{1}{x}\bigg]\bigg)\wedge 0.
\end{equation}
Note that $\Psi(x)<0$ if and only if $x>1$. Finally, we define the function
\begin{equation}\label{Equation Phi function}
\begin{aligned}
\Phi(\theta,\lambda,q)=-S(\theta)&+(1-\theta)\log[1-\pi_1(\lambda\theta)]+\theta\log\pi_1(\lambda\theta)\\
&+(1-\theta)\{\Psi(\tfrac{\lambda(1-\theta)}{q})-(\tfrac{q-1}{2q})\lambda(1-\theta)+\log q\}.
\end{aligned}
\end{equation}
Our main result may then be stated as follows:

\begin{theorem}\label{Theorem Rate function for size of largest component}
Fix $q>0$ and $\lambda>0$. Then, for every $\theta\in[0,1]$,
\begin{equation}\label{Equation Rate function for size of largest component}
\lim_{\epsilon\downarrow 0}\lim_{n\to\infty}\frac{1}{n}\log\phi_{n,\lambda,q}\big[\lvert\mathcal{V}_{\epsilon n}\rvert=\lfloor \theta n \rfloor\big]=\Phi(\theta,\lambda,q)-\sup_{\theta\in[0,1]}\Phi(\theta,\lambda,q).
\end{equation}
\end{theorem}

In the language of large deviations, Theorem \ref{Theorem Rate function for size of largest component} says that in the limit $\epsilon\downarrow 0$, the sequence of random variables $\lvert\mathcal{V}_{\epsilon n}\rvert/n$ satisfies the large deviations principle with rate function $\Phi(\theta,\lambda,q)$. Moreover, if $\theta^*\in[0,1]$ maximises $\Phi(\theta,\lambda,q)$, then Theorem \ref{Theorem Rate function for size of largest component} implies that approximately $\theta^*n$ vertices of $K_n(\omega)$ belong to components of order $n$. The following lemma tells us that in fact, these vertices all belong to a single component of size $\theta^*n$:

\begin{lemma}\label{Lemma Uniqueness of large components}
Fix $q>0$ and $\lambda>0$. Then, for every $\lambda>0$ and $\epsilon>0$ there exists $c=c(\lambda,\epsilon)>0$ such that for every $\theta>\epsilon>0$,
\begin{equation}\label{Equation Uniqueness of large components}
\phi_{n,\lambda,q}[\lvert\mathcal{V}_{\epsilon n}\rvert=\lfloor \theta n \rfloor, \mathcal{N}_{\epsilon n}=1]\geq(1-e^{-cn})\phi_{n,\lambda,q}[\lvert\mathcal{V}_{\epsilon n}\rvert=\lfloor \theta n \rfloor].
\end{equation}
\end{lemma}

It remains to specify the maximiser $\theta^*$. We will see in Section $5$ that the value $\theta^*$ maximising $\Phi(\theta,\lambda,q)$ is given by
\begin{equation}\label{Equation theta lambda}
\theta(\lambda,q)=\begin{cases}0&\text{ if }\lambda<\lambda_c(q)\\\theta_{\text{max}}&\text{ if }\lambda\geq\lambda_c(q)\end{cases}
\end{equation}
where the \textit{critical} value $\lambda_c$ is defined as
\begin{equation}\label{Equation Critical value of lambda}
\lambda_c(q)=\begin{cases}q&\text{ if }q\leq 2\\2\bigg(\frac{q-1}{q-2}\bigg)\log(q-1)&\text{ if }q>2\end{cases}
\end{equation}
and $\theta_{\text{max}}$ is the largest solution of the \textit{mean field equation}
\begin{equation}\label{Equation Mean field}
e^{-\lambda\theta}=\frac{1-\theta}{1+(q-1)\theta}.
\end{equation}
The solutions to (\ref{Equation Mean field}) are discussed in detail in (\cite{BGJ1996}, Lemma $2.5$). In particular, it is shown there that $\theta_{\max}>0$ for $\lambda>\lambda_c$. Consequently, Theorem \ref{Theorem Rate function for size of largest component} implies that the largest component of the graph $K_n(\omega)$ is of order $n$ asymptotically almost surely when $\lambda > \lambda_c$. In this way, we recover the asymptotic phase transition for the size of the largest connected component established in \cite{BGJ1996}. The behaviour for $\lambda=\lambda_c$ is more complicated, and will not be discussed here.

Theorem \ref{Theorem Rate function for size of largest component} was proven in \cite{BCS2007} for the special case of percolation by conditioning on the set $A$ of vertices contained in the largest component and computing the exponential rates of the events that $A$ is connected and that $A^c$ does not contain any large components. We will extend this approach to the random cluster model. To this end, let $K$ be the event that $K_n(\omega)$ is connected, let $B_r$ be the event that $K_n(\omega)$ contains no components of size larger than $r$ and let $L$ be the event that $K_n(\omega)$ is acyclic. The following three theorems,  which generalise (\cite{BCS2007}, Theorems $2.3$, $2.4$ and $2.5$), yield the exponential rates of these three events, beginning with the event $K$:

\begin{theorem}\label{Theorem Weighted rate function for connectedness}
Fix $q>0$ and $\lambda>0$. Then
\begin{equation}\label{Equation Weighted rate function for connectedness}
\lim_{n\to\infty}\frac{1}{n}\log Z_{n,\lambda,q}\big[K\big]=\log \pi_1(\lambda).
\end{equation}
Moreover, convergence is uniform for $\lambda$ belonging to compact subsets of $[0,\infty)$.
\end{theorem}

Observe that the limit obtained in (\ref{Equation Weighted rate function for connectedness}) is independent of $q$. This is no coincidence; on the event $K$, $K_n(\omega)$ is connected and so $k(\omega)=1$. It follows that $Z_{n,\lambda,q}[K]=q\phi_{n,\lambda}[K]$, and so the proof given in (\cite{BCS2007}, Theorem $2.3$) for percolation trivially generalises to the random cluster model.

Next, we compute the exponential rate of the event $L$. On this event, we have the correspondence $k(\omega)=n-\lvert E_n(\omega)\rvert$ between the number of components and edges of the graph $K_n(\omega)$. In particular, it is possible to absorb the cluster weight $q$ into the edge weight and extend (\cite{BCS2007}, Theorem $2.4$) in the following form:

\begin{theorem}\label{Theorem Weighted rate function for acyclic graphs}
Fix $q>0$ and $\lambda>0$. Then
\begin{equation}\label{Equation Weighted rate function for acyclic graphs}
\lim_{n\to\infty}\frac{1}{n}\log Z_{n,\lambda,q}\big[L\big]=\lim_{r\to\infty}\lim_{n\to\infty}\frac{1}{n}\log Z_{n,\lambda,q}\big[B_r\cap L\big]=\Psi(\tfrac{\lambda}{q})-(\tfrac{q-1}{2q})\lambda+\log q.
\end{equation}
Moreover, convergence is uniform for $\lambda$ belonging to compact subsets of $(0,\infty)\setminus\{q\}$.
\end{theorem}

As in \cite{BCS2007}, one may prove that the exponential rate of the event that $K_n(\omega)$ contains only small components coincides with the exponential rate of the event that $K_n(\omega)$ is acyclic. This is not surprising, as in \cite{BGJ1996} it was shown that almost all vertices outside of the giant component belong to trees. This argument is summarised in the following analogue of (\cite{BCS2007}, Theorem $2.5$):

\begin{theorem}\label{Theorem Acyclic equals small components}
Fix $q>0$ and $\lambda>0$. Then
\begin{equation}\label{Equation Acyclic equals small components 1}
\lim_{r\to\infty}\liminf_{n\to\infty}\frac{1}{n}\log Z_{n,\lambda,q}\big[B_r\big]=\lim_{r\to\infty}\limsup_{n\to\infty}\frac{1}{n}\log Z_{n,\lambda,q}\big[B_r\big]=\lim_{n\to\infty}\frac{1}{n}\log Z_{n,\lambda,q}\big[L\big]
\end{equation}
and
\begin{equation}\label{Equation Acyclic equals small components 2}
\lim_{\epsilon\downarrow 0}\liminf_{n\to\infty}\frac{1}{n}\log Z_{n,\lambda,q}\big[B_{\epsilon n}\big]=\lim_{\epsilon\downarrow 0}\limsup_{n\to\infty}\frac{1}{n}\log Z_{n,\lambda,q}\big[B_{\epsilon n}\big]=\lim_{n\to\infty}\frac{1}{n}\log Z_{n,\lambda,q}\big[L\big].
\end{equation}
Moreover, convergence is uniform for $\lambda$ belonging to compact subsets of $(0,\infty)\setminus\{q\}$.
\end{theorem}

Observe that the preceding three theorems compute the exponential rates of the events $K$, $L$ and $B_r$ for the random cluster measure \textit{before} normalisation. Indeed, we will compute the rate function for the size of the largest connected component in the following form:

\begin{theorem}\label{Theorem Weighted rate function for size of largest component}
Fix $q>0$ and $\lambda>0$. Then, for every $\theta\in[0,1]$,
\begin{equation}\label{Equation Rate function for size of largest component}
\lim_{\epsilon\downarrow 0}\lim_{n\to\infty}\frac{1}{n}\log Z_{n,\lambda,q}\big[\lvert\mathcal{V}_{\epsilon n}\rvert=\lfloor \theta n \rfloor\big]=\Phi(\theta,\lambda,q).
\end{equation}
Moreover, convergence is uniform for $\lambda$ belonging to compact subsets of $(0,\infty)\setminus\{q\}$.
\end{theorem}

In order to turn these theorems into statements about probabilities in the random cluster model, we will reintroduce the partition function using the following result:

\begin{theorem}\label{Theorem Free energy of random cluster model on complete graph}
Fix $q>0$ and $\lambda>0$. Then
\begin{equation}\label{Equation Free energy of random cluster model on complete graph}
\lim_{n\to\infty}\frac{1}{n}\log Z_{n,\lambda,q}=\sup_{\theta\in[0,1]}\Phi(\theta,\lambda,q).
\end{equation}
\end{theorem}

The limit of equation (\ref{Equation Free energy of random cluster model on complete graph}) is known as the \textit{free energy} of the random cluster model, and is not a new quantity of interest; it was computed in (\cite{BGJ1996}, Theorem $2.6$). In Lemma \ref{Theorem Restatement of free energy}, we show that our computation agrees with theirs.

The structure of the free energy provides some hint as to its derivation. Indeed, for fixed $\epsilon > 0$, one may decompose the partition function $Z_{n,\lambda,q}$ according to the number of vertices in components of size at least $\epsilon n$ to obtain
\begin{equation}\label{Equation Partition decomposition}
Z_{n,\lambda,q}=\sum_{k=0}^n Z_{n,\lambda,q}\big[\lvert\mathcal{V}_{\epsilon n}\rvert = (\tfrac{k}{n})n\big].
\end{equation}
We seek to apply Theorem \ref{Theorem Weighted rate function for size of largest component} to each of the terms on the right hand side of (\ref{Equation Partition decomposition}). We may then dominate the sum in (\ref{Equation Partition decomposition}) in terms of its largest summand, and apply the \textit{Laplace Principle} when taking the appropriate limit. However, these steps involve some technicalities, which will be discussed in detail in Section $5$. In particular, uniform convergence of the rate function is required in order to pass to a supremum in the limit. This is the reason why it is specified in our theorems, while it was not required for the random graphs in \cite{BCS2007}.

\section{Acyclic graphs}

In this section, we prove Theorems \ref{Theorem Weighted rate function for acyclic graphs} and \ref{Theorem Acyclic equals small components} by analysing the quantity $Z_{n,\lambda,q}[L\cap B_r]$, as was done for percolation in \cite{BCS2007}. Let $m_l$ be the number of connected components of size $l$ in $K_n(\omega)$. Then, the weight of a configuration $\omega\in L\cap B_r$ is given by
\begin{equation}
Z_{n,\lambda,q}[\omega]=q^{\sum m_l}(1-\tfrac{\lambda}{n})^{{n\choose 2}-n+\sum m_l}\prod_{l\geq 1}(\tfrac{\lambda}{n})^{(l-1)m_l}.
\end{equation}
Moreover, conditionally on the numbers $m_1,\cdots,m_n$, there are $\frac{n!}{\prod_lm_l!(l!)^{m_l}}$ ways of choosing appropriately sized components in $K_n(\omega)$. In addition, each component of size $l$ has $a_l=l^{l-2}$ possible spanning trees. Thus
\begin{equation}\label{Equation Acyclic rearrangement 1}
Z_{n,\lambda,q}[L\cap B_r]=\displaystyle\sum_{\substack{\sum lm_l=n \\ m_l=0\forall l>r}}\frac{n!}{\prod_l[m_l!(l!)^{m_l}]}q^{\sum m_l}(1-\tfrac{\lambda}{n})^{{n\choose 2}-n+\sum m_l}\prod_{l\geq 1}\bigg[a_l(\tfrac{\lambda}{n})^{l-1}\bigg]^{m_l}.
\end{equation}
Introducing the sum
\begin{equation}\label{Equation Acyclic rearrangement 2}
Q_{n,k,r}=\displaystyle\sum_{\substack{\sum lm_l=n\\\sum m_l=k\\ m_l=0\forall l>r}}\prod_{l\geq 1}\bigg(\frac{a_l}{l!}\bigg)^{m_l}\frac{1}{m_l!}
\end{equation}
allows us to rewrite (\ref{Equation Acyclic rearrangement 1}) as
\begin{equation}\label{Equation Acyclic rearrangement 3}
Z_{n,\lambda,q}[L\cap B_r]=n!(\tfrac{\lambda}{n})^n(1-\tfrac{\lambda}{n})^{{n\choose 2}-n}\sum_{k=1}^n(\tfrac{\lambda}{n})^{-k}[q(1-\tfrac{\lambda}{n})]^kQ_{n,k,r}.
\end{equation}
The equation (\ref{Equation Acyclic rearrangement 3}) employs the same re-arrangement as that done for percolation in \cite{BCS2007}, with the random cluster model introducing an extra factor of $q^k$ in the summand. We cite the following proposition regarding the asymptotic behaviour of the quantity $Q_{n,k,r}$:

\begin{proposition}[\cite{BCS2007}, Proposition $4.1$]\label{Proposition Large nk behaviour of Q}
Consider the polynomial
\begin{equation}
F_r(s)=\sum_{l=1}^r\frac{s^la_l}{l!}.
\end{equation}
Then for all $n,k,r\geq 1$
\begin{equation}
Q_{n,k,r}\leq\frac{1}{k!}\inf_{s>0}\frac{F_r(s)^k}{s^n}.
\end{equation}
Moreover, for each $\eta>0$ there is an $n_0<\infty$ and a sequence $(c_r)_{r\geq 1}$ of positive numbers such that for all $n\geq n_0$, $k\geq 1$ and $r\geq 2$ such that $k<(1-\eta)n$ and $rk>n(1+\eta)$, we have
\begin{equation}
Q_{n,k,r}\geq\frac{c_r}{\sqrt{n}}\frac{1}{k!}\inf_{s>0}\frac{F_r(s)^k}{s^n}.
\end{equation}
\end{proposition}

Applying Stirling's Formula to the factorial term, we may use Proposition \ref{Proposition Large nk behaviour of Q} to obtain an upper bound on the summand in (\ref{Equation Acyclic rearrangement 3}) of
\begin{equation}
(\tfrac{\lambda}{n})^{-k}[q(1-\tfrac{\lambda}{n})]^kQ_{n,k,r}\leq e^{o(n)}\inf_{s>0}\exp\{n\Theta_r(s,k/n,\lambda/q)\},
\end{equation}
where the error term is bounded uniformly for $\lambda$ belonging to compact subsets of $(0,\infty)$, and the function $\Theta_r(s,\theta,\alpha)$ is given by
\begin{equation}\label{Equation Theta function}
\Theta_r(s,\theta,\alpha)=-\theta\log\alpha-\theta\log\theta+\theta+\theta\log F_r(s)-\log s.
\end{equation}
The notation $\alpha=\lambda/q$ is chosen to agree with \cite{BCS2007}. Notice that a sum of $n$ terms is sandwiched between its maximal term and $n$ times its maximal term. This implies an upper bound on the equation (\ref{Equation Acyclic rearrangement 3}) of
\begin{equation}\label{Equation Acyclic upper bound}
Z_{n,\lambda,q}[L\cap B_r]\leq n!n(\tfrac{\lambda}{n})^n(1-\tfrac{\lambda}{n})^{{n\choose 2}-n} e^{o(n)}\sup_k\inf_{s>0}\exp\{n\Theta_r(s,k/n,\lambda/q)\}
\end{equation}
and a lower bound (for fixed $\eta >0$) on the equation (\ref{Equation Acyclic rearrangement 3}) of
\begin{equation}\label{Equation Acyclic lower bound}
Z_{n,\lambda,q}[L\cap B_r]\geq n!(\tfrac{\lambda}{n})^n(1-\tfrac{\lambda}{n})^{{n\choose 2}-n} e^{o(n)}\sup_k\inf_{s>0}\exp\{n\Theta_r(s,k/n,\lambda/q)\},
\end{equation}
where the supremum is now taken over $k$ satisfying $\tfrac{1}{r}(1+\eta)n<k<(1-\eta)n$. If we can show that the supremum of $\Theta$ is contained in this interval (for sufficiently small $\eta$ and sufficiently large $r$) then the bounds (\ref{Equation Acyclic upper bound}) and (\ref{Equation Acyclic lower bound}) will coincide. This supremum is evaluated in (\cite{BCS2007}, Lemma $4.2$). We reproduce their lemma here, as it will be important to check that any convergence is uniform:

\begin{lemma}[\cite{BCS2007}, Lemma $4.2$]\label{Lemma Maximiser of Theta}
Let $\alpha>0$ and $r\geq 2$. Then there is a unique $(s_r,\theta_r)\in [0,\infty]\times[1/r,1]$ for which
\begin{equation}\label{Equation 1 Maximiser of Theta}
\Theta_r(s_r,\theta_r,\alpha)=\sup_{1/r\leq \theta\leq 1}\inf_{s>0}\Theta_r(s,\theta,\alpha).
\end{equation}
Moreover, $s_r\in(0,\infty)$ satisfies the limit
\begin{equation}\label{Equation 2 Maximiser of Theta}
\lim_{r\to\infty}s_r=\begin{cases}\alpha e^{-\alpha} &\text{ if }\alpha\leq 1\\ \tfrac{1}{e} &\text{ if } \alpha > 1\end{cases}
\end{equation}
and $\theta_r\in(1/r,1)$ satisfies the limit
\begin{equation}\label{Equation 3 Maximiser of Theta}
\lim_{r\to\infty}\theta_r=\begin{cases}1-\tfrac{\alpha}{2} &\text{ if }\alpha\leq 1\\ \tfrac{1}{2\alpha} &\text{ if } \alpha > 1\end{cases}
\end{equation}
Combining (\ref{Equation 2 Maximiser of Theta}) and (\ref{Equation 3 Maximiser of Theta}) yields the limit
\begin{equation}\label{Equation 4 Maximiser of Theta}
\lim_{r\to\infty}\Theta_r(s_r,\theta_r,\alpha)=\begin{cases}1+\tfrac{\alpha}{2}-\log\alpha&\text{ if } \alpha \leq 1\\1+\tfrac{1}{2\alpha} &\text{ if } \alpha > 1\end{cases}
\end{equation}
Moreover, convergence is uniform for $\alpha$ belonging to compact subsets of $(0,\infty)\setminus\{1\}$.
\end{lemma}

Note that we may rewrite the limit (\ref{Equation 4 Maximiser of Theta}) as
\begin{equation}
\lim_{r\to\infty}\Theta_r(s_r,\theta_r,\alpha)=1+\tfrac{\alpha}{2}-\log\alpha+\Psi(\alpha).
\end{equation}
As $\alpha=\lambda/q$, the convergence (in $r$) will be uniform for $\lambda$ belonging to compact subsets of $(0,\infty)\setminus\{q\}$, which is precisely the claim of Theorem \ref{Theorem Weighted rate function for acyclic graphs}.

\begin{proof}
We cite the fact that $(s_r,\theta_r)\in(0,\infty)\times(1/r,1)$ from (\cite{BCS2007}, Lemma $4.2$). By setting the partial derivatives of $\Theta_r$ equal to $0$, we see the maximising pair $(s_r,\theta_r)$ is a solution of the equations
\begin{equation}\label{Equation 1 Proof Maximiser of Theta}
sF_r'(s)=\alpha,\ F_r(s)=\alpha\theta.
\end{equation}
We analyse each of these equations in turn, beginning with the equation $sF_r'(s)=\alpha$. By definition, $sF_r'(s)$ is equal to the sum
\begin{equation}\label{Equation 2 Proof Maximiser of Theta}
\sum_{l=1}^r a_l\frac{s^l}{(l-1)!}=\sum_{l=1}^r\frac{1}{l}\frac{l^l}{l!}s^l.
\end{equation}
Applying Stirling's Formula to the term $l^l/l!$, we see that for $s>1/e$, the sum in (\ref{Equation 2 Proof Maximiser of Theta}) diverges. In the case $s\leq 1/e$, (\ref{Equation 2 Proof Maximiser of Theta}) converges to the function $W(s)$ satisfying $We^{-W}=s$ (a fact which we again cite from \cite{BCS2007}). 

We now split the argument into two cases, supposing first that $\alpha$ belongs to a bounded interval $[\alpha_0,\alpha_1]$ with $\alpha_1<1$. Define
\begin{equation}\label{Equation 3 Proof Maximiser of Theta}
r_1=\inf\{r:\ (1/e)F_r'(1/e)>\alpha_1\},
\end{equation}
and observe that for every $r\geq r_1$ and $\alpha\leq\alpha_1$, we must have $s_r < 1/e$. In particular, the error
\begin{equation}\label{Equation 4 Proof Maximiser of Theta}
\Delta_r(s):=W(s)-sF_r'(s)
\end{equation}
may be uniformly bounded by $\Delta_r:=\Delta_r(1/e)$, which converges to $0$ as $r\to\infty$. As $s_rF_r'(s_r)=\alpha$, we see that $\lvert W(s_r)-\alpha \rvert \leq \Delta_r$. Moreover, by the Mean Value Theorem, we have
\begin{equation}\label{Equation 6 Proof Maximiser of Theta}
\lvert s_r-\alpha e^{-\alpha}\rvert=\lvert W(s_r)e^{-W(s_r)}-\alpha e^{-\alpha}\rvert\leq ce^{-c}\lvert W(s_r)-\alpha\rvert
\end{equation}
for some $c\in(\alpha-\Delta_r,\alpha+\Delta_r)$. As $ce^{-c}\leq 1$, we deduce that $\lvert s_r-\alpha e^{-\alpha}\rvert\leq \Delta_r$, which converges to $0$ uniformly.

Next, suppose that $\alpha \in [\alpha_0,\alpha_1]$ with $\alpha_0>1$. In this case, $s_r$ is bounded below by $1/e$, and we may discard lower order terms in the sum (\ref{Equation 7 Proof Maximiser of Theta}) to obtain the bound
\begin{equation}\label{Equation 7 Proof Maximiser of Theta}
sF_r'(s)\geq r^{-3/2}(es)^r
\end{equation}
If we write $s=\gamma/e$, then (\ref{Equation 7 Proof Maximiser of Theta}) exceeds $\alpha_1$ if
\begin{equation}\label{Equation 8 Proof Maximiser of Theta}
\gamma \geq \exp(\tfrac{1}{r}\log\alpha_1+\tfrac{3}{2r}\log r):=\gamma_r.
\end{equation}
Thus $\lvert s_r-1/3\rvert\leq(\gamma_r-1)/e$, which converges to $0$ uniformly.

It remains to analyse the equation $F_r(s)=\alpha\theta$, separating the argument into the same two cases as before. Suppose first that $\alpha\in[\alpha_0,\alpha_1]$ with $\alpha_1<1$. To find a limit for $\theta_r$, we find a limit for $F_r(s_r)$ by integrating the equation
\begin{equation}\label{Equation 9 Proof Maximiser of Theta}
sF_r'(s)=W(s)-\Delta_r(s).
\end{equation}
Using the identity $We^{-W}=s$ and its derivative $W'(1-W)e^{-W}=1$ with respect to $s$, we obtain the differential equation
\begin{equation}\label{Equation 10 Proof Maximiser of Theta}
F_r'(s)=W'(1-W)-\frac{\Delta_r(s)}{s},
\end{equation}
which has the solution
\begin{equation}\label{Equation 11 Proof Maximiser of Theta}
F_r(s)=W-\tfrac{1}{2}W^2-\int_0^s\tfrac{\Delta_r(t)}{t}dt.
\end{equation}
Plugging in the equations $F_r(s_r)=\alpha\theta_r$ and $W(s_r)=\alpha+\Delta_r(s_r)$ yields
\begin{equation}\label{Equation 12 Proof Maximiser of Theta}
\theta_r=1-\tfrac{1}{2}\alpha+\tfrac{1}{\alpha}\bigg(\Delta_r(s_r)-\alpha\Delta_r(s_r)-\tfrac{1}{2}\Delta_r(s_r)^2-\int_0^{s_r}\tfrac{\Delta_r(t)}{t}dt\bigg).
\end{equation}
As $\alpha$ is bounded away from $0$, $s_r$ is bounded below $1/e$ for $r$ sufficiently large, and $\Delta_r(t)/t$ converges to $0$ as $r\to\infty$, the error term in (\ref{Equation 12 Proof Maximiser of Theta}) converges to $0$ uniformly.

Finally, suppose $\alpha\in[\alpha_0,\alpha_1]$ with $\alpha_0>1$, and recall that in this case, $s_r$ converges to $e^{-1}$ uniformly. By the triangle inequality, we have
\begin{equation}
\lvert F_r(s_r)-\tfrac{1}{2}\rvert\leq\lvert F_r(s_r)-F_r(e^{-1})\rvert+\lvert F_r(e^{-1})-\tfrac{1}{2}\rvert.
\end{equation}
The second term converges to $0$ independently of $\alpha$. For the first, we apply the Mean Value Theorem to obtain
\begin{equation}
\lvert F_r(s_r)-F_r(e^{-1})\rvert\leq F_r'(c)\lvert s_r-e^{-1}\rvert,
\end{equation}
for some $c\in(e^{-1},s_r)$. Noting that $F_r'(s)$ is increasing and $F_r'(s_r)=\alpha/s_r$ is bounded, we deduce that the first term must also converge uniformly. The result follows after substituting $F_r(s_r)=\alpha\theta_r$.
\end{proof}

Note that the supremum in Lemma \ref{Lemma Maximiser of Theta} is taken over the interval $\theta\in[1/r,1]$, whereas the sums in the bounds of equations (\ref{Equation Acyclic upper bound}) and (\ref{Equation Acyclic lower bound}) are over discrete subsets of the interval $[1/r,1]$. We claim that the supremums over these two sets coincide in the limit as $n\to\infty$. Let $\theta_{r,n}=\tfrac{1}{n}\lfloor \theta_r n\rfloor$ and define $s_{r,n}$ to be the number $s$ satisfying
\begin{equation}
\Theta_r(s_{r,n},\theta_{r,n},\alpha)=\inf_{s>0}\Theta_r(s,\theta_{r,n},\alpha).
\end{equation}
It is sufficient to prove that the pair $(s_{r,n},\theta_{r,n})$ converges to the pair $(\theta_r, s_r)$ as $n\to\infty$:

\begin{lemma}\label{Lemma Convergence to continuous supremum}
Fix $\alpha>0$ and $r\geq 2$. Then
\begin{equation}
\lim_{n\to\infty}\Theta_r(s_{r,n},\theta_{r,n},\alpha)=\Theta_r(s_r,\theta_r,\alpha).
\end{equation}
Moreover, convergence is uniform for $\alpha$ in compact subsets of $(0,\infty)\setminus\{1\}$.
\end{lemma}

To prove Lemma \ref{Lemma Convergence to continuous supremum}, we will need the following proposition:

\begin{proposition}\label{Proposition Polynomial Quotient}
Let $P(x)=\sum_{r=0}^na_rx^r$ be a polynomial with non-negative co-efficients, and consider the function
\begin{equation}
Q(x)=\frac{P(x)}{xP'(x)}.
\end{equation}
Then for $x>0$, $Q(x)$ is decreasing.
\end{proposition}

\begin{proof}
Differentiation yields
\begin{equation}
Q'(x)=\frac{1}{x^2P'(x)^2}\bigg(xP'(x)^2-P(x)P'(x)-xP(x)P''(x)\bigg).
\end{equation}
It will be sufficient to show that
\begin{equation}
x^3P'(x)^2Q'(x)=\sum_{s=0}^n\sum_{t=0}^n[sta_sa_t-ta_sa_t-t(t-1)a_sa_t]x^{s+t},
\end{equation}
is negative. Writing $m=s+t$ (and noting that the sum is the same when interchanging the order of summation of $s$ and $t$) this may be rewritten as
\begin{align}
x^3P'(x)^2Q(x)&=\tfrac{1}{2}\sum_{m=0}^{2n}\sum_{r=0}^n[2r(m-r)-m-r(r-1)-(m-r)(m-r-1)]a_ra_{m-r}x^m\nonumber\\
&=-\tfrac{1}{2}\sum_{m=0}^{2n}\sum_{r=0}^n(m-2r)^2a_ra_{m-r}x^m
\end{align}
where we have set $a_l=0$ for any $l<0$. This is negative, as required.
\end{proof}

\begin{proof}[Proof of Lemma \ref{Lemma Convergence to continuous supremum}]
Recall that
\begin{equation}\label{Equation 1 Proof Convergence to continuous supremum}
\Theta_r(s,\theta,\alpha)=-\theta\log\alpha-\theta\log\theta+\theta+\theta\log F_r(s)-\log s.
\end{equation}
Using the triangle inequality, we may bound the difference between $\Theta_r(s_{r,n},\theta_{r,n},\alpha)$ and $\Theta_r(s_r,\theta_r,\alpha)$ by the sum of the differences of the terms in (\ref{Equation 1 Proof Convergence to continuous supremum}). If we can show that each of these differences converges to $0$, we will be done:
\begin{enumerate}
\item First, consider the term $\lvert 1-\log\alpha\rvert \lvert \theta_r-\theta_{r,n}\rvert$, and note that  $\lvert \theta_r-\theta_{r,n}\rvert\leq 1/n$ by definition. As the function $1-\log\alpha$ is uniformly bounded on compact subsets of $\alpha\in(0,\infty)\setminus\{1\}$, this term converges uniformly to $0$ as $n\to\infty$.
\item Next, consider the term $\lvert \theta_r\log\theta_r-\theta_{r,n}\log\theta_{r,n}\rvert$. By the triangle inequality, we have
\begin{equation}\label{Equation 2 Proof Convergence to continuous supremum}
\lvert \theta_r\log\theta_r-\theta_{r,n}\log\theta_{r,n}\rvert\leq\lvert \theta_r(\log\theta_r-\log\theta_{r,n})\rvert+\lvert (\theta_r-\theta_{r,n})\log\theta_{r,n}\rvert,
\end{equation}
where both of the terms in the right hand side converge uniformly to $0$ by uniform continuity of the logarithm away from $0$.
\item Next, consider the term $\lvert\log s_r - \log s_{r,n}\rvert$. Recall that, at a stationary point, $\theta$ is given by
\begin{equation}
\theta=\frac{F_r(s)}{sF_r'(s)}.
\end{equation}
By Proposition \ref{Proposition Polynomial Quotient}, $\theta$ is decreasing when viewed as a function of $s$. Fix $a<b$ such that, on our chosen compact subset of $\alpha\in(0,\infty)\setminus\{1\}$,
\begin{equation}
\theta(b)<\theta_{r,n}<\theta_r<\theta(a).
\end{equation}
On $[a,b]$, $\theta(s)$ is continuous and injective. In particular, it has a uniformly continuous inverse $s=h(\theta)$ on the compact set $[\theta(b),\theta(a)]$, and so $s_r-s_{r,n}$ converges uniformly to $0$. As the logarithm is uniformly continuous on intervals bounded away from $0$, the same holds for $\log s_r-\log s_{r,n}$.
\item Finally, consider the term $\lvert\theta_r\log F_r(s_r)-\theta_{r,n}\log F_r(s_{r,n})\rvert$. By the triangle inequality, we have
\begin{align}
\lvert\theta_r\log F_r(s_r)-\theta_{r,n}\log F_r(s_{r,n})\rvert\leq&\lvert\theta_r\log F_r(s_r)-\theta_{r}\log F_r(s_{r,n})\rvert\nonumber\\
&+\lvert\theta_r\log F_r(s_{r,n})-\theta_{r,n}\log F_r(s_{r,n})\rvert.
\end{align}
The second term converges to $0$ as $n\to\infty$ by continuity of $F_r$. For the first term, we observe that
\begin{equation}
\log F_r(s_{r,n})-\log F_r(s_r)=\int_{s_r}^{s_{r,n}}\frac{F_r'(t)}{F_r(t)}dt \leq r\int_{s_r}^{s_{r,n}}t \ dt \leq rs_{r,n}(s_{r,n}-s_r).
\end{equation}
For $n$ sufficiently large, $s_{r,n}$ is uniformly bounded and so this converges to uniformly.
\end{enumerate}
\end{proof}

We have established that $\theta_r$ converges to a limit uniformly for $\lambda$ belonging to compact subsets of $(0,\infty)\setminus\{q\}$. Moreover, this limit belongs to a compact subset of $(0,1)$. In particular, one may choose $\eta>0$ sufficiently small and $r>0$, $n>0$ sufficiently large to ensure that $\theta_{r,n}n$ belongs to the interval $[\tfrac{1}{r}(1+\eta)n, (1-\eta)n]$. Then, we may apply Proposition \ref{Proposition Large nk behaviour of Q} to the bounds of equations  (\ref{Equation Acyclic upper bound}) and (\ref{Equation Acyclic lower bound}) to obtain
\begin{equation}\label{Equation Penultimate Z equality}
Z_{n,\lambda,q}[L\cap B_r]=e^{o(n)}\exp\{n(-1-\tfrac{1}{2}\lambda+\log\lambda+\Theta_r(s_{r,n},\theta_{r,n},\lambda/q))\},
\end{equation}
where the $o(n)$ term is bounded uniformly for $\lambda$ belonging to compact subsets of $(0,\infty)\setminus\{q\}$.

\begin{proof}[Proof of Theorem \ref{Theorem Weighted rate function for acyclic graphs}]
From (\ref{Equation Penultimate Z equality}), we have that
\begin{align}\label{Equation 1 Proof Weighted rate function for acyclic graphs}
\frac{1}{n}\log Z_{n,\lambda,q}[L\cap B_r] &=\Theta_r(s_{r,n},\theta_{r,n},\lambda/q)-\Theta_r(s_r,\theta_r,\lambda/q)\nonumber\\
&+\Theta_r(s_r,\theta_r,\lambda/q)-1-\tfrac{1}{2}\lambda+\log\lambda\nonumber\\
&+\tfrac{o(n)}{n}.
\end{align}
Take the limit of (\ref{Equation 1 Proof Weighted rate function for acyclic graphs}) as $n\to\infty$ and $r\to\infty$ in that order. The first line converges to $0$ by Lemma \ref{Lemma Convergence to continuous supremum}. The second line converges to the required rate by Lemma \ref{Lemma Maximiser of Theta}.
\end{proof}

Next, we prove Theorem \ref{Theorem Acyclic equals small components}. We need one preliminary lemma, corresponding to (\cite{BCS2007}, Lemma $5.1$):

\begin{lemma}\label{Lemma Small components and acyclic graphs}
Fix $q>0$ and $\lambda>0$. Then
\begin{equation}\label{Equation Lemma Small components and acyclic graphs}
Z_{n,\lambda,q}[B_r]\leq Z_{n,\lambda,q}[L](1-\tfrac{\lambda}{n})^{-\tfrac{1}{2}rn}.
\end{equation}
\end{lemma}

\begin{proof}
We follow the proof of (\cite{BCS2007}, Lemma $5.1$), showing that
\begin{equation}\label{Equation 1 Small components and acyclic graphs}
\phi_{n,\lambda,q}[B_r]\leq \phi_{n,\lambda,q}[L](1-\tfrac{\lambda}{n})^{-\tfrac{1}{2}rn}.
\end{equation}
Given a set of vertices $S\subset \{1,\cdots,n\}$, let $C_S$ denote the restriction of the graph $K_n(\omega)$ to $S$, and let $T$ be a tree on $S$. Conditionally on the event $C_S \supset T$, all vertices of $S$ belong to the same component of $K_n(\omega)$. In particular, any edge in $E(S)\setminus T$ is open independently with probability $\lambda/n$, and so
\begin{equation}\label{Equation 2 Small components and acyclic graphs}
\frac{\phi_{n,\lambda,q}[C_S=T]}{\phi_{n, \lambda,q}[C_S\supset T]}=(1-\tfrac{\lambda}{n})^{{|S|\choose 2}-|S|+1}\geq (1-\tfrac{\lambda}{n})^{\tfrac{1}{2}|S|^2}.
\end{equation}
Let $K_S$ be the event that $C_S$ is connected. Then
\begin{equation}\label{Equation 3 Small components and acyclic graphs}
\phi_{n,\lambda,q}[K_S]\leq \sum_T\phi_{n,\lambda,q}[C_S\supset T]\leq  (1-\tfrac{\lambda}{n})^{-\tfrac{1}{2}|S|^2}\phi_{n,\lambda,q}[C_S\text{ is a tree}].
\end{equation}
Now, let $L_r$ denote the event that each component of $K_n(\omega)$ is either acyclic or has size at most $r$, and note that $B_r\subset L_r$. Let $\{S_j\}$ be a partition of $\{1,\cdots,n\}$ and let $\phi_{n,\lambda,q}[\{S_j\}]$ denote the probability that $\{S_j\}$ are the connected components of $K_n(\omega)$. Conditioning the event $L_r$ on the partition $\{S_j\}$ of connected components, we have
\begin{equation}\label{Equation 4 Small components and acyclic graphs}
\phi_{n,\lambda,q}[L_r]=\sum_{\{S_j\}}\phi_{n,\lambda,q}[\{S_j\}]\phi_{n,\lambda,q}[L_r\mid\{S_j\}].
\end{equation}
Moreover, conditionally on the partition $\{S_j\}$ of connected components, the states of edges in different components are independent. In particular, we may write

\begin{align*}
\phi_{n,\lambda,q}[L_r|\{S_j\}]&=\prod_{j:\lvert S_j\rvert >r}\phi_{n,\lambda,q}[C_{S_j}\text{ is a tree}\mid K_{S_j}]\\
&=\prod_j\phi_{n,\lambda,q}[C_{S_j}\text{ is a tree}\mid K_{S_j}]\prod_{j:\lvert S_j\rvert\leq r}\phi_{n,\lambda,q}[C_{S_j}\text{ is a tree}\mid K_{S_j}]^{-1}\\
&\leq \prod_j\phi_{n,\lambda,q}[C_{S_j}\text{ is a tree}\mid K_{S_j}]\prod_{j:|S_j|\leq r}(1-\tfrac{\lambda}{n})^{-\tfrac{1}{2}\lvert S_j\rvert^2}
\end{align*}
where the inequality in the third line is a consequence of (\ref{Equation 2 Small components and acyclic graphs}). As $\lvert S_j\rvert<r$ for every term in the second product and the sum over $\lvert S_j\rvert$ is at most $n$, we obtain
\begin{equation}
\phi_{n,\lambda,q}[L_r\mid\{S_j\}]\leq\phi_{n,\lambda,q}[L\mid\{S_j\}](1-\tfrac{\lambda}{n})^{-\tfrac{1}{2}rn}.
\end{equation}
Finally, we see that
\begin{align*}
\phi_{n,\lambda,q}[B_r]&\leq\phi_{n,\lambda,q}[L_r]\\
&=\sum_{\{S_j\}}\phi_{n, \lambda,q}[\{S_j\}]\phi_{n,\lambda,q}[L_r\mid\{S_j\}]\\
&\leq\sum_{\{S_j\}}\phi_{n,\lambda,q}[\{S_j\}]\phi_{n,\lambda,q}[L\mid\{S_j\}](1-\tfrac{\lambda}{n})^{-\tfrac{1}{2}rn}\\
&=\phi_{n,\lambda,q}[L](1-\tfrac{\lambda}{n})^{-\tfrac{1}{2}rn}
\end{align*}
which establishes (\ref{Equation 1 Small components and acyclic graphs}). We obtain (\ref{Equation Lemma Small components and acyclic graphs}) by multiplying both sides of  (\ref{Equation 1 Small components and acyclic graphs}) by the partition function.
\end{proof}

\begin{proof}[Proof of Theorem \ref{Theorem Acyclic equals small components}]
The theorem is a consequence of the following inequalities:
\begin{equation}\label{Equation Four inequalities}
\begin{aligned}
\lim_{n\to\infty}\frac{1}{n}\log Z_{n,\lambda,q}[L]&\leq\lim_{r\to\infty}\liminf_{n\to\infty}\frac{1}{n}\log Z_{n,\lambda,q}[B_r],\\
\lim_{r\to\infty}\liminf_{n\to\infty}\frac{1}{n}\log Z_{n,\lambda,q}[B_r]&\leq\lim_{\epsilon\downarrow 0}\liminf_{n\to\infty}\frac{1}{n}\log Z_{n,\lambda,q}[B_{\epsilon n}],\\
\lim_{r\to\infty}\limsup_{n\to\infty}\frac{1}{n}\log Z_{n,\lambda,q}[B_r]&\leq\lim_{\epsilon\downarrow 0}\limsup_{n\to\infty}\frac{1}{n}\log Z_{n,\lambda,q}[B_{\epsilon n}],\\
\lim_{\epsilon\downarrow 0}\limsup_{n\to\infty}\frac{1}{n}\log Z_{n,\lambda,q}[B_{\epsilon n}] &\leq \lim_{n\to\infty}\frac{1}{n}\log Z_{n,\lambda,q}[L].
\end{aligned}
\end{equation}
To prove the first inequality, we apply the inclusion $B_r \supset B_r\cap L$ and Theorem \ref{Theorem Weighted rate function for acyclic graphs} to see that
\begin{equation}
\liminf_{n\to\infty}\frac{1}{n}\log Z_{n,\lambda,q}[B_r]\geq\liminf_{n\to\infty}\frac{1}{n}Z_{n,\lambda,q}[B_r\cap L]\to\lim_{n\to\infty}\frac{1}{n}\log Z_{n,\lambda,q}[L]
\end{equation}
as $r\to\infty$. To prove the second inequality, fix $r\geq 2,\ \epsilon > 0$, and let $N=\lceil r/\epsilon \rceil$. Then, for every $n\geq N$, we have $\epsilon n \geq \epsilon\lceil r/\epsilon \rceil \geq r.$ As a result, $B_{\epsilon n} \supset B_r$, and
\begin{equation}
\frac{1}{n}\log Z_{n,\lambda,q}[B_r]\leq \frac{1}{n}\log Z_{n,\lambda,q}[B_{\epsilon n}].
\end{equation}
One may take the infimum over $m\geq n$, followed by the limit as $n\to\infty$, on both sides to obtain
\begin{equation}
\liminf_{n\to\infty}\frac{1}{n}\log Z_{n,\lambda,q}[B_r]\leq\liminf_{n\to\infty}\frac{1}{n}\log Z_{n,\lambda,q}[B_{\epsilon n}].
\end{equation}
As $r$ and $\epsilon$ were arbitrary, taking the limits $r\to\infty$ and $\epsilon\downarrow 0$ yields the second inequality. The proof of the third inequality is similar. To prove the fourth inequality, we apply Lemma \ref{Lemma Small components and acyclic graphs} to obtain
\begin{equation}
\frac{1}{n}\log Z_{n,\lambda,q}[B_{\epsilon n}] \leq \frac{1}{n}\log Z_{n,\lambda,q}[L]+\tfrac{1}{2}\lambda\epsilon,
\end{equation}
from which the inequality follows after taking the limit superior as $n\to\infty$ and the limit as $\epsilon\downarrow 0$.
\end{proof}

\section{Uniqueness and the largest component}

In this section, we prove Theorem \ref{Theorem Weighted rate function for size of largest component} and Lemma \ref{Lemma Uniqueness of large components}. The proof of the latter is identical to that of (\cite{BCS2007}, Lemma $6.2$), and requires only the following analogue of (\cite{BCS2007}, Lemma $6.1$), which estimates the probability of the event $K_{\epsilon,2}$ that $K_n(\omega)$ is connected or has exactly two connected components each of size at least $\epsilon n$:

\begin{lemma}[\cite{BCS2007}, Lemma $6.1$]\label{Lemma Uniqueness prerequisite}
Fix $q>0$. Then for all $\lambda_0>0$ and $\epsilon_0>0$ there exists $c_1=c_1(\lambda_0,\epsilon_0)>0$ such that for all $\epsilon\geq\epsilon_0$ and $\lambda\leq\lambda_0$, we have
\begin{equation}
\limsup_{n\to\infty}\frac{1}{n}\log\phi_{n,\lambda,q}[K^c|K_{\epsilon,2}]<-c_1.
\end{equation}
\end{lemma}

\begin{proof}
Observe that
\begin{equation}\label{Equation 1 Lemma Uniqueness prerequisite}
\phi_{n,\lambda,q}[K^c|K_{\epsilon,2}]=\frac{\phi_{n,\lambda,q}[K_{\epsilon,2}\setminus K]}{\phi_{n,\lambda,q}[K_{\epsilon,2}\setminus K]+\phi_{n,\lambda,q}[K]} \leq \frac{\phi_{n,\lambda,q}[K_{\epsilon,2}\setminus K]}{\phi_{n,\lambda,q}[K]}.
\end{equation}
It suffices to show that the ratio on the right hand side of (\ref{Equation 1 Lemma Uniqueness prerequisite}) decays to zero exponentially in $n$, with a rate that is uniformly bounded in $\epsilon\geq\epsilon_0$ and $\lambda\leq\lambda_0$. Observe that $\omega \in K_{\epsilon,2}\setminus K$ if and only if we may find a set $A\subset K_n$ (where $A$ depends on $\omega$) of vertices of size between $\epsilon n$ and $n-\epsilon n$ such that $A$, $A^c$ are connected components of $K_n(\omega)$ and there are no open edges between them. We count the configurations satisfying these conditions. Let $E(A,A^c)$ be the set of edges between $A$ and its complement in $K_n(\omega)$, and suppose that $\vert A\rvert=k$. It can be seen that $A$ is disconnected from $A^c$ in $K_n(\omega)$ with probability
\begin{equation}
\phi_{n,\lambda,q}[E(A,A^c)=\emptyset]=\frac{Z_{k,\lambda k/n,q}Z_{n-k, \lambda(1-k/n),q}}{Z_{n,\lambda,q}}(1-\lambda/n)^{k(n-k)}.
\end{equation}
Conditionally on the above event, $A$ is connected in $K_n(\omega)$ with probability
\begin{equation}
\phi_{k,\lambda k/n,q}[K]=Z_{k,\lambda k/n,q}[K]/Z_{k,\lambda k/n,q},
\end{equation}
and $A^c$ is connected in $K_n(\omega)$ with probability
\begin{equation}
\phi_{n-k, \lambda(1-k/n),q}[K]=Z_{n-k, \lambda(1-k/n),q}[K]/Z_{n-k, \lambda(1-k/n),q}.
\end{equation}
Note that there are ${n\choose k}$ choices for $A$, and that we have counted any pair $(A,A^c)$ twice. Thus, we have the equation
$$
Z_{n,\lambda,q}[K_{\epsilon,2}\setminus K]=\frac{1}{2}\sum_{\epsilon n\leq k \leq n-\epsilon n}{n\choose k}\bigg(1-\frac{\lambda}{n}\bigg)^{k(n-k)}Z_{k,\lambda k/n,q}[K]Z_{n-k, \lambda(1-k/n),q}[K].$$
Applying Theorem \ref{Theorem Weighted rate function for connectedness}, we obtain
\begin{equation}\label{Equation 6 Lemma Uniqueness prerequisite}
\frac{\phi_{n,\lambda,q}[K_{\epsilon,2}\setminus K]}{\phi_{n,\lambda,q}[K]}=e^{o(n)}\sum_{\epsilon n\leq k \leq n-\epsilon n}{n\choose k}\frac{\pi_1(\lambda\tfrac{k}{n})^k\pi_1(\lambda(1-\tfrac{k}{n}))^{n-k}}{\pi_1(\lambda)^n}\bigg(1-\frac{\lambda}{n}\bigg)^{k(n-k)},
\end{equation}
where $\pi_1(\lambda)=1-e^{-\lambda}$, and the $e^{o(n)}$ term is the error term found in Theorem \ref{Theorem Weighted rate function for connectedness}. The lemma now concludes identically to  (\cite{BCS2007}, Lemma $6.1$), rewriting (\ref{Equation 6 Lemma Uniqueness prerequisite}) as
\begin{equation}\label{Equation 8 Lemma Uniqueness prerequisite}
\frac{\phi_{n,\lambda,q}[K_{\epsilon,2}\setminus K]}{\phi_{n,\lambda,q}[K]}=e^{o(n)}\sum_{\epsilon n\leq k \leq n-\epsilon n}e^{n[\Xi(k/n)-\Xi(0)]},
\end{equation}
where the function $\Xi$ is defined by
\begin{equation}\label{Equation 9 Lemma Uniqueness prerequisite}
\Xi(\theta)= -S(\theta)+\theta\log\pi_1(\lambda\theta)+(1-\theta)\log\pi_1(\lambda(1-\theta)) -\lambda\theta(1-\theta).
\end{equation}
We now bound the sum in (\ref{Equation 8 Lemma Uniqueness prerequisite}) by $n$ times its maximal summand. As $\Xi$ is convex and symmetric around the point $1/2$, the summand is maximised at the endpoints. In particular, we have the bound
\begin{equation}
\frac{\phi_{n,\lambda,q}[K_{\epsilon,2}\setminus K]}{\phi_{n,\lambda,q}[K]}\leq e^{o(n)}e^{n[\Xi(\epsilon)-\Xi(0)]}.
\end{equation}
More explicitly, we may take any value $c_1 <\Xi(0) -\Xi(\epsilon)$ provided that $n$ is sufficiently large.
\end{proof}

\begin{proof}[Proof of Lemma \ref{Lemma Uniqueness of large components}]
The proof is identical to that of (\cite{BCS2007}, Lemma $6.2$). In particular, it will suffice to prove that
\begin{equation}\label{Equation 1 Proof Uniqueness of large components}
\phi_{n,\lambda,q}[\lvert\mathcal{V}_{\epsilon n}\rvert=\lfloor \theta n \rfloor,\ \mathcal{N}_{\epsilon n}>1]\leq e^{-cn}\phi_{n,\lambda,q}[\lvert\mathcal{V}_{\epsilon n}\rvert=\lfloor \theta n \rfloor].
\end{equation}
For a given vertex $x$, let $\mathcal{C}_x$ denote the component of $K_n(\omega)$ containing $x$. On the event $\{\lvert\mathcal{V}_{\epsilon n}\rvert=\lfloor \theta n \rfloor,\mathcal{N}_{\epsilon n}>1\}$, we may find a pair of vertices $x,y\in[n]$ in $K_n(\omega)$ such that $\lvert\mathcal{C}_x\rvert\geq \epsilon n$, $\lvert\mathcal{C}_y\rvert\geq \epsilon n$ and $x\nleftrightarrow y$. Define the following two events for the random graph:
\begin{align*}
A_1&=\{\lvert\mathcal{C}_x\rvert\geq \epsilon n\}\cap\{\lvert\mathcal{C}_y\rvert\geq \epsilon n\}\cap\{x\nleftrightarrow y\},\\
A_2&=\{\lvert\mathcal{C}_x\rvert\geq \epsilon n\}\cap\{\lvert\mathcal{C}_y\rvert\geq \epsilon n\}.
\end{align*}
As the complete graph is transitive, the probabilities of the events $A_1$ and $A_2$ do not depend on the particular choices of $x$ and $y$. By the union bound, it follows that
\begin{equation}\label{Equation 2 Proof Uniqueness of large components}
\phi_{n,\lambda,q}[\lvert\mathcal{V}_{\epsilon n}\rvert=\lfloor \theta n \rfloor,\mathcal{N}_{\epsilon n}>1]\leq n^2\phi_{n,\lambda,q}[\{\lvert\mathcal{V}_{\epsilon n}\rvert=\lfloor \theta n \rfloor\} \cap A_1].
\end{equation}
We now condition further on the set $\mathcal{C}_x\cup\mathcal{C}_y$. For a given set $\mathcal{C}\subset [n]$, let $D$ be the event that $\mathcal{C}$ is disjoint from $\mathcal{C}^c$ in $K_n(\omega)$ and that its complement contains $\lfloor \theta n \rfloor -\lvert\mathcal{C}\rvert$ vertices in components of size at least $\epsilon n$. Then
\begin{align*}
\phi_{n,\lambda,q}[\{\lvert\mathcal{V}_{\epsilon n}\rvert=\lfloor \theta n \rfloor\} \cap A_1]&=\sum_{\mathcal{C}\subset[n]}\phi_{n,\lambda,q}[A_1\cap\{\mathcal{C}_x\cup\mathcal{C}_y=\mathcal{C}\}\cap D],\\
&=\sum_{\mathcal{C}\subset[n]}\phi_{n,\lambda,q}[A_1\cap\{\mathcal{C}_x\cup\mathcal{C}_y=\mathcal{C}\}\mid D]\phi_{n,\lambda,q}[D].
\end{align*}
Write $m=\lvert\mathcal{C}\rvert$, $\tilde{\lambda}=\lambda\theta$, and $\tilde{\epsilon}=\epsilon/\theta$. On the event $D$, the random cluster measure restricts to the random cluster measure $\phi_{\theta n,\tilde{\lambda},q}$ on $\mathcal{C}$. Moreover, we have the following correspondences between events:
\begin{align*}
A_1\cap \{\mathcal{C}_x\cup\mathcal{C}_y=\mathcal{C}\}&=K^c\cap K_{\tilde{\epsilon},2},\\
A_2\cap \{\mathcal{C}_x\cup\mathcal{C}_y=\mathcal{C}\}&=K_{\tilde{\epsilon},2}.
\end{align*}
As $\tilde{\epsilon}\geq \epsilon$ for every $\theta > 0$, we may apply Lemma \ref{Lemma Uniqueness prerequisite} to deduce that
\begin{equation}\label{Equation 3 Proof Uniqueness of large components}
\phi_{n,\lambda,q}[A_1\cap\{\mathcal{C}_x\cup\mathcal{C}_y=\mathcal{C}\}\mid D]\leq e^{-c_1m}\phi_{n,\lambda,q}[A_2\cap\{\mathcal{C}_x\cup\mathcal{C}_y=\mathcal{C}\}\mid D],
\end{equation}
which allows us to rewrite (\ref{Equation 2 Proof Uniqueness of large components}) as
\begin{equation}\label{Equation 4 Proof Uniqueness of large components}
\phi_{n,\lambda,q}[\lvert\mathcal{V}_{\epsilon n}\rvert=\lfloor \theta n \rfloor, \mathcal{N}_{\epsilon n}>1]\leq n^2e^{-c_1m}\phi_{n,\lambda,q}[\{\lvert\mathcal{V}_{\epsilon n}\rvert=\lfloor \theta n \rfloor\} \cap A_2].
\end{equation}
The result follows as $\{\lvert\mathcal{V}_{\epsilon n}\rvert=\lfloor \theta n \rfloor\} \cap A_2 \subset \{\lvert\mathcal{V}_{\epsilon n}\rvert=\lfloor \theta n \rfloor\}$ and $m\geq \epsilon n$.
\end{proof}

\begin{proof}[Proof of Theorem \ref{Theorem Weighted rate function for size of largest component}]
By Lemma \ref{Lemma Uniqueness of large components}, it is sufficient to prove that
\begin{equation}
\lim_{\epsilon\downarrow 0}\lim_{n\to\infty}\frac{1}{n}\log Z_{n,\lambda,q}[\lvert\mathcal{V}_{\epsilon n}\rvert=\lfloor \theta n \rfloor, \mathcal{N}_{\epsilon n}=1]=\Phi(\theta,\lambda,q).
\end{equation}
The case $\theta=1$ reduces to Theorem \ref{Theorem Weighted rate function for connectedness} and the case $\theta=0$ reduces to Theorems \ref{Theorem Weighted rate function for acyclic graphs} and \ref{Theorem Acyclic equals small components}. Let $\theta \in(0,1)$, $\epsilon\in(0,\theta)$ and assume that $\theta n$ is an integer. Given a configuration $\omega$, observe that $\omega\in\{\lvert\mathcal{V}_{\epsilon n}\rvert = \theta n, N_{\epsilon n}=1\}$ if and only if we may find a subset $A \subset K_n$ of vertices (where $A$ depends on $\omega$) of size $\theta n$ such that $A$ is a connected component of $K_n(\omega)$, $A^c$ contains no connected components of $K_n(\omega)$ of size exceeding $\epsilon n$, and $E(A,A^c)=0$. We count the possible configurations which satisfy these conditions. Note that there are ${n\choose\theta n}$ possible choices for $A$, and $A$ is disconnected from $A^c$ in $K_n(\omega)$ with probability
\begin{equation}
\phi_{n,\lambda,q}[E(A,A^c)=0]=\frac{Z_{\theta n,\lambda\theta,q}Z_{(1-\theta)n,\lambda(1-\theta),q}}{Z_{K_n,p,q}}(1-\lambda/n)^{\theta(1-\theta)n^2}.
\end{equation}
Conditionally on this event, $A$ is connected in $K_n(\omega)$ with probability
\begin{equation}
\phi_{\theta n,\lambda\theta,q}[K]=Z_{\theta n,\lambda\theta,q}[K]/Z_{\theta n,\lambda\theta,q},
\end{equation}
and $A^c$ does not contain any components of size exceeding $\epsilon n$ in $K_n(\omega)$ with probability
\begin{equation}
\phi_{(1-\theta)n,\lambda(1-\theta),q}[B_{\epsilon n}]=Z_{(1-\theta)n,\lambda(1-\theta),q}[B_{\epsilon n}]/Z_{(1-\theta)n,\lambda(1-\theta),q}.
\end{equation}
Thus, we have the equation
\begin{equation}
Z_{n,\lambda,q}[\lvert\mathcal{V}_{\epsilon n}\rvert= \theta n,N_{\epsilon n}=1]={n\choose\theta n}(1-\tfrac{\lambda}{n})^{\theta n(1-\theta)n}Z_{\theta n,\lambda\theta,q}[K]Z_{(1-\theta)n,\lambda(1-\theta),q}[B_{\epsilon n}].
\end{equation}
We now take logarithms, divide by $n$ and take the limit as $n\to\infty$. Note first that
\begin{equation}
\lim_{n\to\infty}\frac{1}{n}\log\bigg({n\choose\theta n}(1-\lambda/n)^{\theta n(1-\theta)n}\bigg) = -S(\theta)+(1-\theta)\log[1-\pi_1(\lambda\theta)].
\end{equation}
Applying Theorem \ref{Theorem Weighted rate function for connectedness}, we also have
\begin{equation}
\lim_{n\to\infty}\frac{1}{n}\log Z_{\theta n,\lambda\theta,q}[K]=\theta\log \pi_1(\lambda\theta).
\end{equation}
Finally, we apply Theorems \ref{Theorem Weighted rate function for acyclic graphs} and \ref{Theorem Acyclic equals small components} to obtain
\begin{equation}
\begin{aligned}
\lim_{\epsilon\downarrow 0}\lim_{n\to\infty}\frac{1}{n}\log Z_{(1-\theta)n,\lambda(1-\theta),q}[B_{\epsilon n}]=&(1-\theta)\times\\
&\bigg\{\Psi(\tfrac{\lambda(1-\theta)}{q})-(\tfrac{q-1}{2q})\lambda(1-\theta)+\log q\bigg\}.
\end{aligned}
\end{equation}
Summing these limits yields the result. Convergence is uniform as each of the individual limits converge uniformly.
\end{proof}

\section{Free energy}

In this section, we prove Theorem \ref{Theorem Free energy of random cluster model on complete graph}. To begin, we let $\epsilon>0$. Then, we may decompose the partition function as
\begin{equation}\label{Equation 1 Free energy of random cluster model on complete graph proof}
Z_{n,\lambda,q}=Z_{n,\lambda,q}[B_{\epsilon n}]+\sum_{k>\epsilon n} Z_{n,\lambda,q}[\lvert \mathcal{V}_{\epsilon n}\rvert = k].
\end{equation}
By Lemma \ref{Lemma Uniqueness of large components}, we may write
\begin{equation}\label{Equation 2 Free energy of random cluster model on complete graph proof}
Z_{n,\lambda,q}=Z_{n,\lambda,q}[B_{\epsilon n}]+(1-o(1))\sum_{k>\epsilon n} Z_{n,\lambda,q}[\lvert \mathcal{V}_{\epsilon n}\rvert = k,\ \mathcal{N}_{\epsilon n}=1].
\end{equation}
We aim to apply Theorem \ref{Theorem Weighted rate function for size of largest component} to each summand in (\ref{Equation 2 Free energy of random cluster model on complete graph proof}). Recall that
\begin{equation}\label{Equation 3 Free energy of random cluster model on complete graph proof}
Z_{n,\lambda,q}[\lvert\mathcal{V}_{\epsilon n}\rvert= k,N_{\epsilon n}=1]={n\choose k}(1-\tfrac{\lambda}{n})^{k(n-k)}Z_{k,\lambda k/n,q}[K]Z_{n-k,\lambda(1-k/n),q}[B_{\epsilon n}].
\end{equation}
In order to apply Theorem \ref{Theorem Weighted rate function for size of largest component} to all of the summands in (\ref{Equation 2 Free energy of random cluster model on complete graph proof}) simultaneously, we require that the term $\lambda(1-k/n)$ belongs to a compact subset of $(0,\infty)\setminus\{q\}$. It will suffice to prove that the terms for which $k/n$ is close to $1$ or $1-q/\lambda$ have negligible probability, which we do via the following two tail inequalities for sufficiently small $\epsilon$:
\begin{equation}\label{Equation 4 Free energy of random cluster model on complete graph proof}
\begin{aligned}
Z_{n,\lambda,q}[\lvert\mathcal{V}_{\epsilon n}\rvert \geq (1-\epsilon)n,\ \mathcal{N}_{\epsilon n}=1]&\leq o(1) Z_{n,\lambda,q},\\
Z_{n,\lambda,q}[\lvert\mathcal{V}_{\epsilon n}\rvert \leq (1-q/\lambda+\epsilon)n,\ \mathcal{N}_{\epsilon n}=1]&\leq o(1)Z_{n,\lambda,q}.
\end{aligned}
\end{equation}
Equivalently, we show that
\begin{equation}\label{Equation 5 Free energy of random cluster model on complete graph proof}
\begin{aligned}
\phi_{n,\lambda,q}[\lvert\mathcal{V}_{\epsilon n}\rvert \geq (1-\epsilon)n,\ \mathcal{N}_{\epsilon n}=1]&\leq o(1),\\
\phi_{n,\lambda,q}[\lvert\mathcal{V}_{\epsilon n}\rvert \leq (1-q/\lambda+\epsilon)n,\ \mathcal{N}_{\epsilon n}=1]&\leq o(1).
\end{aligned}
\end{equation}
Both inequalities in (\ref{Equation 5 Free energy of random cluster model on complete graph proof}) may be proven via direct comparisons with percolation, using the following law of large numbers for the size of the largest component of $K_n(\omega)$ under the percolation measure $\phi_{n,\lambda}$, cited from \cite{H2016}:

\begin{theorem}[\cite{H2016}, Theorem $4.8$]\label{Theorem LLN percolation giant component}
Fix $\lambda>1$ and let $p=\lambda/n$. Then, for every $\nu\in(\tfrac{1}{2},1)$ there exists $\delta=\delta(\lambda,\nu)>0$ such that
\begin{equation}
\phi_{n,\lambda}[\lvert \lvert \mathcal{C}_1 \rvert - \theta(\lambda,1)n\rvert \geq n^\nu] = O(n^{-\delta}).
\end{equation}
\end{theorem}

We begin with the first inequality of (\ref{Equation 5 Free energy of random cluster model on complete graph proof}). Note that the random cluster measure $\phi_{n,\lambda,q}$ is stochastically dominated by the percolation measure $\phi_{n,\lambda}$ for $q\geq 1$ (see e.g. (\cite{G2006}, Theorem $3.21$)), yielding the upper bound
\begin{equation}\label{Equation 6 Free energy of random cluster model on complete graph proof}
\phi_{n,\lambda,q}[\lvert\mathcal{V}_{\epsilon n}\rvert \geq (1-\epsilon)n,\ \mathcal{N}_{\epsilon n}=1]\leq\phi_{n,\lambda}[\lvert\mathcal{V}_{\epsilon n}\rvert \geq (1-\epsilon)n,\ \mathcal{N}_{\epsilon n}=1].
\end{equation}
As the percolation measure is stochastically ordered in the edge weight $p$, we may assume that $\lambda >1$, in which case we are done by Theorem \ref{Theorem LLN percolation giant component} provided we take $\epsilon < 1-\theta(\lambda,1)$.

We now turn to the second inequality of (\ref{Equation 5 Free energy of random cluster model on complete graph proof}), noting first that it is only relevant if $\lambda>q$. If this is the case, then the random cluster measure $\phi_{n,\lambda,q}$ stochastically dominates the supercritical percolation measure $\phi_{n,\lambda/q}$, again by (\cite{G2006}, Theorem $3.21$). As a result, Theorem \ref{Theorem LLN percolation giant component} may be applied to show that
\begin{equation}\label{Equation Largest component LLN}
\phi_{n,\lambda,q}[\lvert\mathcal{V}_{\epsilon n}\rvert \leq (\theta(\lambda/q,1)-\epsilon)n,\ \mathcal{N}_{\epsilon n}=1]=O(n^{-\delta}).
\end{equation}
Write $\alpha=\lambda/q$, as before. We claim that $\theta(\alpha,1)>1-\alpha^{-1}$. As $\theta(\alpha,1)$ solves the equation
\begin{equation}
\alpha=-\frac{1}{\theta(\alpha,1)}\log(1-\theta(\alpha,1))
\end{equation}
it will be sufficient to prove that
\begin{equation}
-\frac{1}{\theta(\alpha,1)}\log(1-\theta(\alpha,1)) < (1-\theta(\alpha,1))^{-1},
\end{equation}
which is a consequence of the inequality $(1-x)\log(1-x)+x>0$ for $x \in(0,1)$. In particular, we are done provided we take $\epsilon<\theta(\lambda/q,1)-1+q/\lambda$.

Applying both inequalities of (\ref{Equation 4 Free energy of random cluster model on complete graph proof}) to (\ref{Equation 2 Free energy of random cluster model on complete graph proof}), we see that
\begin{equation}\label{Free energy sum}
(1-o(1))Z_{n,\lambda,q}=Z_{n,\lambda,q}[B_{\epsilon n}]+\sum_{k >\max\{\epsilon, \theta(\lambda/q,1)-\epsilon\}}^{(1-\epsilon)n}Z_{n,\lambda,q}[\lvert \mathcal{V}_{\epsilon n}\rvert = k,\ \mathcal{N}_{\epsilon n}=1].
\end{equation}
We are finally in a position to prove Theorem \ref{Theorem Free energy of random cluster model on complete graph}:

\begin{proof}[Proof of Theorem \ref{Theorem Free energy of random cluster model on complete graph}]
Let $s_n$ be the maximal summand of Equation (\ref{Free energy sum}). As the summands converge uniformly, so too does the maximal summand, with limit
\begin{equation}
\lim_{\epsilon\downarrow 0}\lim_{n\to\infty}\frac{1}{n}\log s_n = \sup_{\theta > \theta(\lambda/q,1)}\Phi(\theta,\lambda,q).
\end{equation}
Moreover, as the sum is bounded between $s_n$ and $ns_n$, we have
\begin{equation}
\frac{1}{n}\log s_n\leq\frac{1}{n}\log((1-o(1))Z_{n,\lambda,q})\leq \frac{1}{n}\log s_n+\frac{1}{n}\log n.
\end{equation}
Taking the limits as $n\to\infty$ and $\epsilon\downarrow 0$ in that order yields the result.
\end{proof}

We have proven that the free energy of the random cluster model converges to the supremum of the function $\Phi(\theta,\lambda,q)$. Finally, we evaluate this supremum. In particular, the following lemma shows that our computation of the free energy agrees with the one found in (\cite{BGJ1996}, Theorem $2.6$):

\begin{lemma}\label{Theorem Restatement of free energy}
Let $q>0$ and $\lambda>0$. Then
\begin{equation}\label{Equation Restatement of free energy}
\sup_{\theta\in[0,1]}\Phi(\theta,\lambda,q)=\sup_{\theta > \theta(\lambda/q,1)}\Phi(\theta,\lambda,q)=\frac{g(\theta(\lambda,q))}{2q}-\bigg(\frac{q-1}{2q}\bigg)\lambda+\log q
\end{equation}
where the function $g:(0,1)\to\mathbb{R}$ is defined by
\begin{equation}\label{Equation Free energy functional}
g(\theta)=-(q-1)(2-\theta)\log(1-\theta)-[2+(q-1)\theta]\log[1+(q-1)\theta].
\end{equation}
\end{lemma}

\begin{proof}
We separate the argument into the cases $\lambda(1-\theta)>q$ and $\lambda(1-\theta)<q$, corresponding to the regions in which the $\Psi$ function is defined. We use the shorthand notation
\begin{equation}\label{Equation 1 Restatement of free energy}
a=\frac{1-\theta}{q\theta},\ b = \frac{e^{-\lambda\theta}}{1-e^{-\lambda\theta}},\ k = \lambda\theta.
\end{equation}
When $\lambda(1-\theta)>q$, the derivative of $\Phi$ with respect to $\theta$ is given by
\begin{equation}\label{Equation 2 Restatement of free energy}
\frac{\partial}{\partial\theta}\Phi(\theta,\lambda,q)=(kb)-\log(kb)-1.
\end{equation}
As $x-\log x-1 \geq 0$ with equality if and only if $x=1$, the derivative in (\ref{Equation 2 Restatement of free energy}) is equal to zero only if $kb=1$. This is equivalent to the equation $1+\lambda\theta=e^{\lambda\theta}$, for which the only solution is $\theta=0$. When $\lambda(1-\theta)<q$, we obtain the derivative
\begin{equation}\label{Equation 3 Restatement of free energy}
\frac{\partial}{\partial\theta}\Phi(\theta,\lambda,q) = (\log a - ka) - (\log b - kb).
\end{equation}
The function $\log x - kx$ is convex, with a maximum at $x=\tfrac{1}{k}$. We know that $a\leq\tfrac{1}{k}$ by assumption, and $b\leq\tfrac{1}{k}$ is a consequence of the inequality $1+\lambda\theta \leq e^{\lambda\theta}$. As a result, the derivative in (\ref{Equation 3 Restatement of free energy}) is equal to zero only if $a=b$, which may be rearranged to see that the maximising value $\theta^*$ satisfies the mean field equation (\ref{Equation Mean field}). Conversely, any solution $\theta$ to the mean-field equation satisfies the assumption (and so is a stationary point), as
\begin{equation}\label{Equation 4 Restatement of free energy}
\frac{\lambda(1-\theta)}{q}= ka = kb \leq 1.
\end{equation}
We may now assume $\theta^*$ satisfies the mean-field equation. Under this assumption, one may rewrite $\Phi(\theta^*, \lambda, q)$ in the form
\begin{equation}\label{Equation 5 Restatement of free energy}
\Phi(\theta^*, \lambda, q)=\frac{1}{2q}g(\theta^*)-\frac{q-1}{2q}\lambda+\log q.
\end{equation}
This is the form taken in (\cite{BGJ1996}, Theorem $2.6$). It remains to show that (\ref{Equation 5 Restatement of free energy}) is maximised when we take the solution $\theta(\lambda,q)$ to the mean field equation. We quote the following properties of the function $g$ from \cite{BGJ1996}:
\begin{equation}\label{Equation 6 Restatement of free energy}
\begin{aligned}
g(0)&=g'(0)=0,\\
g''(\theta)&=-\frac{q(q-1)[q-2-2(q-1)\theta]\theta}{(1-\theta)^2[1+(q-1)\theta]^2}.
\end{aligned}
\end{equation}
For $q\leq 2$, $g(\theta)$ is a convex, increasing function and the result is clear. For $q>2$, $g(\theta)$ is initially decreasing. Moreover, $g''(\theta)$ has a zero at $\theta=\tfrac{q-2}{2(q-1)}$, and is increasing thereafter. In particular, $g(\theta)$ is convex for $\theta>\tfrac{q-2}{2(q-1)}$ and has only one zero in this region, which we may compute as $\theta_c=\tfrac{q-2}{q-1}$. Note that $\theta_c$ is the largest solution to the mean-field equation for $\lambda=\lambda_c$.

We claim that $\theta_{\text{max}}$ is increasing as a function of $\lambda$. If $\theta_{\text{max}}(\lambda)=0$ then this is obvious, so we may assume that $\theta_{\text{max}}(\lambda)>0$. Let $\epsilon>0$, and define the function
\begin{equation}
h(\theta):=e^{-(\lambda+\epsilon)\theta}-\frac{1-\theta}{1+(q-1)\theta}.
\end{equation}
Noting that $h(\theta_{\text{max}}(\lambda))<0$ and $h(1)>0$, it follows that $h$ has a zero in the interval $(\theta_{\text{max}}(\lambda),1)$ i.e. $\theta_{\text{max}}(\lambda+\epsilon)>\theta_{\text{max}}(\lambda)$. 

We may now conclude. If $\lambda < \lambda_c$, then $\theta_{\text{max}}(\lambda)<\theta_{\text{max}}(\lambda_c)$ and so $g(\theta_{\text{max}}(\lambda))<0$. In particular, $\theta^*=0$ maximises the free energy. Conversely, if $\lambda > \lambda_c$ then it follows that $g(\theta_{\text{max}}(\lambda))>0$. As $g(\theta)$ is convex for $\theta > \tfrac{1}{2}\theta_c$, it follows that $\theta_{\text{max}}$ is the solution maximising the function $g(\theta)$, and so $\theta^*=\theta_{\text{max}}$.
\end{proof}

\section*{Acknowledgements}

This research was supported by the EPSRC grant EP/N509796/1/1935605. The author would also like to thank Roman Kotecký for his advice and
encouragement and acknowledge the support by the GA\v{C}R grant 20-08468S during his visit to Prague in January 2020. 

\bibliographystyle{plainnat}

\end{document}